\documentclass[11pt,reqno]{amsart}

\usepackage{color}
\usepackage[colorlinks=true, allcolors=blue]{hyperref}
\usepackage{amsmath, amssymb, amsthm}
\usepackage{mathrsfs}
\usepackage{mathtools}
\usepackage[noabbrev,capitalize,nameinlink]{cleveref}
\crefname{equation}{}{}
\usepackage{fullpage}
\usepackage[noadjust]{cite}
\usepackage{graphics}
\usepackage{pifont}
\usepackage{tikz}
\usepackage{bbm}
\usepackage[T1]{fontenc}

\usetikzlibrary{arrows.meta}

\usepackage{environ}
\usepackage{framed}
\usepackage{url}
\usepackage[linesnumbered,ruled,vlined]{algorithm2e}
\usepackage[noend]{algpseudocode}
\usepackage[labelfont=bf]{caption}
\usepackage{cite}
\usepackage{framed}
\usepackage[framemethod=tikz]{mdframed}
\usepackage{appendix}
\usepackage{graphicx}
\usepackage[textsize=tiny]{todonotes}
\usepackage{tcolorbox}
\usepackage{enumerate}
\allowdisplaybreaks[1]

\apptocmd{\sloppy}{\hbadness 10000\relax}{}{} 

\crefname{algocf}{Algorithm}{Algorithms}

\crefname{equation}{}{} 
\AtBeginEnvironment{appendices}{\crefalias{section}{appendix}} 

\usepackage[color,final]{showkeys} 

\colorlet{refkey}{orange!20}
\colorlet{labelkey}{blue!30}

\crefname{algocf}{Algorithm}{Algorithms}

\numberwithin{equation}{section}
\newtheorem{theorem}{Theorem}[section]
\newtheorem{proposition}[theorem]{Proposition}
\newtheorem{lemma}[theorem]{Lemma}

\crefname{claim}{Claim}{Claims}

\newtheorem*{question*}{Question}

\theoremstyle{definition}
\newtheorem{definition}[theorem]{Definition}

\newtheorem*{definition*}{Definition}

\theoremstyle{remark}
\newtheorem*{remark}{Remark}

\newcommand{\snorm}[1]{\lVert#1\rVert}

\newcommand{\sang}[1]{\langle #1 \rangle}

\newcommand{\mb}{\mathbb}

\newcommand{\mbm}{\mathbbm}
\newcommand{\mc}{\mathcal}
\newcommand{\mf}{\mathfrak}

\newcommand{\on}{\operatorname}
\newcommand{\wh}{\widehat}
\newcommand{\wt}{\widetilde}

\allowdisplaybreaks

\title{On the smallest singular value of symmetric random matrices}

\author[A1]{Vishesh Jain}
\address{Simons Institute for the Theory of Computing,
Berkeley, CA 94720, USA}
\email{visheshj@stanford.edu}

\author[A2]{Ashwin Sah}
\author[A3]{Mehtaab Sawhney}
\address{Department of Mathematics, Massachusetts Institute of Technology, Cambridge, MA 02139, USA}
\email{\{asah,msawhney\}@mit.edu}

\begin{document}
\begin{abstract}
We show that for an $n\times n$ random symmetric matrix $A_n$, whose entries on and above the diagonal are independent copies of a sub-Gaussian random variable $\xi$ with mean $0$ and variance $1$, 
\[\mb{P}[s_n(A_n) \le \epsilon/\sqrt{n}] \le O_{\xi}(\epsilon^{1/8} + \exp(-\Omega_{\xi}(n^{1/2}))) \quad \text{for all } \epsilon \ge 0.\]
This improves a result of Vershynin, who obtained such a bound with $n^{1/2}$ replaced by $n^{c}$ for a small constant $c$, and $1/8$ replaced by $(1/8) + \eta$ (with implicit constants also depending on $\eta > 0$). 
Furthermore, when $\xi$ is a Rademacher random variable, we prove that 
\[\mb{P}[s_n(A_n) \le \epsilon/\sqrt{n}] \le O(\epsilon^{1/8} + \exp(-\Omega((\log{n})^{1/4}n^{1/2}))) \quad \text{for all } \epsilon \ge 0.\]
The special case $\epsilon = 0$ improves a recent result of Campos, Mattos, Morris, and Morrison, which showed that $\mb{P}[s_n(A_n) = 0] \le O(\exp(-\Omega(n^{1/2}))).$

The main innovation in our work are new notions of arithmetic structure -- the Median Regularized Least Common Denominator and the Median Threshold, which we believe should be more generally useful in contexts where one needs to combine anticoncentration information of different parts of a vector. 
\end{abstract}

\maketitle

\section{Introduction}\label{sec:introduction}

Let $M_{n}$ denote an $n\times n$ random matrix, each of whose entries is an independent copy of a sub-Gaussian random variable $\xi$ with mean $0$ and variance $1$. Prominent well-studied examples include the Ginibre ensemble (corresponding to $\xi = \mc{N}(0,1)$) and i.i.d.~Rademacher matrices (corresponding to the Rademacher random variable $\xi = \pm 1$ with probability $1/2$ each). 

A landmark result of Rudelson and Vershynin \cite{rudelson2008littlewood} shows that there are absolute constants $C, c > 0$, depending only on the sub-Gaussian norm of $\xi$, for which
\begin{align}
\label{eqn:RV}
    \mb{P}[s_n(M_n) \le \epsilon/\sqrt{n}] \le C \epsilon + 2e^{-cn} \quad \text{for all } \epsilon \ge 0,
\end{align}
where $s_n(M_n) = \inf_{v \in \mb{S}^{n-1}}\snorm{Mv}_2$ denotes the smallest singular value of $M_n$.
Up to the constants $C, c > 0$, the above result is optimal, as can be seen by considering the two examples mentioned above. In particular, this result shows that the probability that an i.i.d.~Rademacher matrix is singular is at most $2\exp(-cn)$ (for some $c > 0$), thereby recovering (and substantially generalising) a well-known result of Kahn, Koml\'os, and Szemer\'edi \cite{kahn1995probability}. We remark that after a series of intermediate works \cite{bourgain2010singularity, tao2007singularity, tao2006random}, a breakthrough result of Tikhomirov \cite{Tik20} established that the probability of singularity of an i.i.d.~Rademacher matrix is at most $(1/2 + o_n(1))^{n}$, which is optimal up to the $o_n(1)$ term. 

In this paper, we will be concerned with $n\times n$ \emph{symmetric} random matrices $A_{n}$ i.e. $(A_{n})_{ij} = (A_{n})_{ji}$, each of whose entries on and above the diagonal is an independent copy of a sub-Gaussian random variable $\xi$ with mean $0$ and variance $1$. We note that the identical distribution assumption may be significantly relaxed (in particular, allowing for the diagonal entries to have a different distribution), although for the sake of simplicity, we do not deal with this modification here; the interested reader is referred to \cite{Ver14} and \cite{livshyts2019smallest}.

While symmetric matrices are especially convenient to work with linear algebraically, the lack of independence between the entries of $A_{n}$ makes the non-asymptotic study of its smallest singular value considerably more challenging than that of $M_{n}$. In the early 1990s, it was conjectured by Weiss that $A_n(\on{Rad})$ (i.e.\ $A_n$ where $\xi$ is a Rademacher random variable) is invertible with probability $1-o_n(1)$. This was only resolved in 2005 by Costello, Tao, and Vu \cite{CTV06}, despite the corresponding statement for $M_n$ (due to Koml\'os \cite{komlos1967determinant}) having been established almost 40 years prior. 

Vershynin \cite{Ver14} showed that for any sub-Gaussian random variable $\xi$ with mean $0$ and variance $1$, there are constants $c, C_\eta$ depending only on the sub-Gaussian norm of $\xi$ such that
\begin{align}
\label{eqn:Ver}
\mb{P}[s_n(A_n)\le \epsilon/\sqrt{n}] \le C_\eta\epsilon^{1/8+\eta} + 2e^{-n^{c}}.
\end{align}
This improves (and generalizes) the nearly concurrent estimate of $O_{C}(n^{-C})$ on the singularity probability of $A_n(\on{Rad})$ obtained by Nguyen \cite{Ngu12} using a novel quadratic variant of the inverse Littlewood--Offord theory. We note that in a subsequent work \cite{nguyen2012least}, Nguyen obtained estimates on the lower tail of $s_n(A_n)$ for a large class of random variables $\xi$, including those not covered by \cite{Ver14}, although the quantitative bounds in this work are much weaker than \cref{eqn:Ver}. 

Recently, the upper bound on the singularity probability of $A_n(\on{Rad})$ has been improved in a couple of works. Building on novel combinatorial techniques in \cite{FJLS2018}, it was shown by Ferber and Jain \cite{ferber2019singularity} that this probability is at most $\exp(-\Omega(n^{1/4}\sqrt{\log{n}}))$. Subsequently, using a different combinatorial method inspired by the method of hypergraph containers \cite{balogh2018method}, Campos, Mattos, Morris, and Morrison improved the bound to $\exp(-\Omega(\sqrt{n}))$. We note that both of these works deal only with $A_n(\on{Rad})$, and only with the singularity probability as opposed to quantitative estimates on $s_n(A_n(\on{Rad}))$.  

The first main result of this paper is a strengthening of \cref{eqn:Ver}; the quantitative bounds are sufficiently powerful to generalize the aforementioned result of Campos et al.~ to all sub-Gaussian random variables. 
\begin{theorem}\label{thm:main}
Let $A_n$ denote an $n\times n$ random symmetric matrix, each of whose entries on and above the diagonal is an independent copy of a sub-Gaussian random variable $\xi$ with mean $0$ and variance $1$. Then, there are constants $C_{\ref{thm:main}}, c_{\ref{thm:main}}$ depending only on the sub-Gaussian norm of $\xi$ such that, for all $\epsilon \ge 0$, 
\[\mb{P}[s_n(A_n)\le\epsilon/\sqrt{n}]\le C_{\ref{thm:main}}\epsilon^{1/8}+2e^{-c_{\ref{thm:main}}n^{1/2}}.\]
\end{theorem}

Next, we consider the particularly well studied case $\xi = \on{Rad}$; setting $\epsilon = 0$ in the theorem below improves the result of Campos et al.~(see (3) in the Remark below).
\begin{theorem}\label{thm:main-2}
Let $A_n$ denote an $n\times n$ random symmetric matrix, each of whose entries on and above the diagonal is an independent Rademacher random variable. Then, there are absolute constants $C_{\ref{thm:main-2}}, c_{\ref{thm:main-2}}$ such that, for all $\epsilon \ge 0$, 
\[\mb{P}[s_n(A_n)\le\epsilon/\sqrt{n}]\le C_{\ref{thm:main-2}}\epsilon^{1/8}+2e^{-c_{\ref{thm:main-2}}n^{1/2}(\log n)^{1/4}}.\]
\end{theorem}
\begin{remark}
(1) We note that \cref{thm:main-2} can be extended to the setting of discrete random variables covered in recent work of the authors \cite{JSS20discrete2}. We leave the details to an interested reader.

(2) The $\epsilon^{1/8}$ term on the right hand side in \cref{thm:main} improves on the $\epsilon^{1/8 + \eta}$ in \cite{Ver14}. It is believed that the correct dependence on $\epsilon$ is $O(\epsilon)$, which would be optimal in light of the Gaussian example.

(3) The term $\exp(-\Omega(n^{1/2}))$ on the right hand side in \cref{thm:main} extends the result of Campos et al.~to general sub-Gaussian random variables, whereas \cref{thm:main-2} improves this result by a factor of $(\log n)^{1/4}$ in the exponent, in the special case when $\xi = \on{Rad}$. A well-known conjecture is that one should be able to replace this with $\exp(-\Omega(n))$, although this will likely require significant new ideas. Indeed, as can be seen from our proof (see also the discussion in \cite[Section~2.2]{campos2019singularity}), $\exp(-\wt\Omega (n^{1/2}))$ is a natural barrier for techniques based on combining uniform anticoncentration estimates for a symmetric matrix row-vector product with tensorization.   
\end{remark}

The main innovation in our work are new notions of arithmetic structure of vectors, which we call the Median Regularized Least Common Denominator (MRLCD) (see \cref{sec:MRLCD}) and the Median Threshold (see \cref{sec:median-threshold}). Compared to the Regularized Least Common Denominator (RLCD) introduced in \cite{Ver14}, and its natural threshold analogue, the MRLCD and median threshold are able to exploit the information that many different projections of a vector are arithmetically unstructured in a simple and transparent manner. Moreover, we are able to show that level sets of the MRLCD and median threshold admit sufficiently small nets at the appropriate scale -- for the MRLCD, this follows by suitably adapting by-now standard bounds due to Rudelson and Vershynin \cite{rudelson2008littlewood}, whereas for the median threshold, we adapt work of Tikhomirov \cite{Tik20} on the singularity of i.i.d~Bernoulli random matrices. As the details are anyway short, we defer further discussion to \cref{sec:MRLCD,sec:median-threshold}. 

We note that since its first appearance in \cite{Ver14}, the RLCD has been used in many works (see, e.g., \cite{nguyen2017random, wei2017investigate, luh2018sparse, lopatto2019tail, o2016conjecture}); the MRLCD (and median threshold, for discrete distributions) can replace these applications in a black-box manner, and likely lead to improved quantitative estimates. We also note that a related use of combinatorially incorporating arithmetic unstructure of different projections of a vector appeared in recent work of the authors \cite{JSS20}; however, the interaction with both the net and anticoncentration estimates is more delicate here.

\subsection{Notation}\label{sub:notation}
We will drop the dimension in the subscript, henceforth denoting $A_n$ by $A$, and denoting its rows by $A_1,\dots,A_n$. For an integer $N$, $\mb{S}^{N-1}$ denotes the set of unit vectors in $\mb{R}^{N}$, and $\mb{B}_2^N$ denotes the unit ball in $\mb{R}^{N}$ (i.e., the set of vectors of Euclidean norm at most $1$). $\snorm{\cdot}_2$ denotes the standard Euclidean norm of a vector, and for a matrix $A = (a_{ij})$, $\snorm{A}$ is its spectral norm (i.e., $\ell^{2} \to \ell^{2}$ operator norm), and $\snorm{A}_{\on{HS}}$ is its Hilbert-Schmidt norm, defined by $\snorm{A}_{\on{HS}}^{2} = \sum_{i,j}a_{ij}^{2}$.

We will let $[N]$ denote the interval $\{1,\dots, N\}$, $\mf{S}_{[N]}$ denote the set of permutations of $[N]$, and $\binom{[N]}{k}$ denote the set of subsets of $[N]$ of size exactly $k$. We will denote multisets by $\{\{\}\}$, so that $\{\{a_1,\dots, a_{n}\}\}$, with the $a_i$'s possibly repeated, is a multi-set of size $n$. For a vector $v \in \mb{R}^{N}$ and $T\subseteq [N]$, $v|_{T}$ denotes the $|T|$-dimensional vector obtained by only retaining the coordinates of $v$ in $T$. We write $u\parallel v$ for $u,v\in\mb{R}^N$ if there is $t\in\mb{R}$ such that $u = tv$ or $tu = v$.

We will also make use of asymptotic notation. For functions $f,g$, $f = O_{\alpha}(g)$ (or $f\lesssim_{\alpha} g$ means that $f \le C_\alpha g$, where $C_\alpha$ is some constant depending on $\alpha$; $f = \Omega_{\alpha}(g)$ (or $f \gtrsim_{\alpha} g$) means that $f \ge c_{\alpha} g$, where $c_\alpha > 0$ is some constant depending on $\alpha$, and $f = \Theta_{\alpha}(g)$ means that both $f = O_{\alpha}(g)$ and $f = \Omega_{\alpha}(g)$ hold.

All logarithms are natural, unless indicated otherwise, and floors and ceilings are omitted when they make no essential difference.

\subsection{Acknowledgements}
We thank Roman Vershynin for comments on the manuscript. The last two authors were supported by the National Science Foundation Graduate Research Fellowship under Grant No.~1745302.

\section{Preliminaries}\label{sec:prelims}
We will need the decomposition of the unit sphere into compressible and incompressible vectors, as formalized by Rudelson and Vershynin \cite{rudelson2008littlewood}.
\begin{definition}[Compressible and incompressible vectors]\label{def:compressible}
For $c_0, c_1 \in (0,1)$, $\on{Comp}(c_0, c_1)$ consists of all vectors $v\in \mb{S}^{n-1}$ which are within Euclidean distance $c_1$ of some vector $w \in \mb{R}^{n}$ satisfying $|\on{Supp}(w)| \le c_0 n$. Moreover, $\on{Incomp}(c_0, c_1):= \mb{S}^{n-1}\setminus \on{Comp}(c_0, c_1)$.  
\end{definition}
In order to prove \cref{thm:main}, it suffices to analyze $\inf_{x \in \on{Incomp(c_0,c_1)}}\snorm{Ax}_{2}$ due to the following.
\begin{lemma}\label{lem:compressible-singular-vector}
There exist $c_0,c_1,c\in(0,1)$ depending only on the sub-Gaussian moment of $\xi$ so that for any vector $u\in\mb{R}^n$, we have
\[\mb{P}\bigg[\inf_{v\in\on{Comp}(c_0,c_1)}\snorm{Av-u}_2 < c\sqrt{n}\bigg]\le 2\exp(-cn).\]
\end{lemma}
\begin{proof}
This follows immediately by combining \cite[Proposition~4.2]{Ver14} with the concentration of the operator norm of random matrices with independent, uniformly sub-Gaussian centered entries (cf.~\cite[Lemma~2.3]{Ver14}).
\end{proof}

\begin{lemma}[Incompressible vectors are spread, cf.~{\cite[Lemma~3.8]{Ver14}}]
\label{lem:spread}
For every $c_0, c_1 \in (0,1)$, we can choose $c_{\ref{lem:spread}} := c_{\ref{lem:spread}}(c_0, c_1) \in (0,1/5)$ depending only on $c_0, c_1$ such that the following holds. For every $v \in \on{Incomp}(c_0, c_1)$, there are at least $2\lceil c_{\ref{lem:spread}}n \rceil $ indices $k \in [n]$ such that 
\[\frac{c_1}{\sqrt{2n}} \le |v_k| \le \frac{1}{\sqrt{c_0 n}}.\]
\end{lemma}

\begin{definition}[Spread set]\
For every $c_0, c_1 \in (0,1)$, and for every $v \in \on{Incomp}(c_0, c_1)$, we assign a subset $\on{Spread}(v) \subseteq [n]$ such that
\[|\on{Spread}(v)| = \lceil c_{\ref{lem:spread}}n \rceil, \emph{ and }\]
\[\frac{c_1}{\sqrt{2n}} \le |v_k| \le \frac{1}{\sqrt{c_0 n}} \text{ for all }k\in\on{Spread}(v).\]
\end{definition}
\begin{definition}\label{def:lambda-spread}
For every $c_0, c_1 \in (0,1)$ and for  $\lambda\in(0,c_{\ref{lem:spread}}/2)$, let $c_{\ref{def:lambda-spread}}(\lambda)n$ be the largest multiple of $\lceil\lambda n\rceil$ less than or equal to $\lceil c_{\ref{lem:spread}}n\rceil$. Note that $c_{\ref{def:lambda-spread}}(\lambda)\ge\frac{c_{\ref{lem:spread}}}{2}$. 

To every $v \in \on{Incomp}(c_0, c_1)$, we assign $\on{Spread}_\lambda(v)\subseteq\on{Spread}(v)$ such that 
\[|\on{Spread}_\lambda(v)| = c_{\ref{def:lambda-spread}}(\lambda)n,\]
and choose a partition
\[\on{Spread}_\lambda(v) = \bigsqcup_{j=1}^k\on{Spread}_\lambda^j(v)\]
into $k = c_{\ref{lem:spread}}(\lambda)n/\lfloor\lambda n\rfloor$ disjoint subsets of size $\lfloor\lambda n\rfloor$. We further assume that the choice of $\on{Spread}_{\lambda}(v)$ and $\on{Spread}_{\lambda}^{j}(v)$ is uniform for a given choice of $\lambda$ and $\on{Spread}(v)$ (in particular, these choices do \emph{not} depend directly on $v$). 
\end{definition}

We recall the definition of the L\'evy concentration function.
\begin{definition}\label{def:levy-concentration}
For a random variable $X$ and $\epsilon \ge 0$, the L{\'e}vy concentration of $X$ of width $\epsilon$ is 
\[\mc{L}(X,\epsilon) = \sup_{x\in\mb{R}}\mb{P}[|X-x|\le\epsilon].\]
\end{definition}

We will also need a slight variant of the standard tensorization lemma, whose proof follows from the usual argument (cf.~\cite[Lemma~2.2]{rudelson2008littlewood}). We include the details for completeness. 
\begin{lemma}[Tensorization]\label{lem:tensorization}
Let $X = (X_1,\dots,X_N)$ be a random vector in $\mb{R}^{N}$ with independent coordinates. Suppose that for all $k \in [N]$, there exist $a_k,b_k \ge 0$ such that
\[\sup_{X_1,\ldots,X_{k-1}}\mc{L}(X_k|X_1,\ldots,X_{k-1}, \epsilon) \le a_k\epsilon + b_k \quad \emph{for all }\epsilon \ge 0.\]
Then
\[\mc{L}(X,\epsilon\sqrt{N}) \le e^N\prod_{k=1}^N(a_k\epsilon + b_k).\]
\end{lemma}
\begin{proof}
We have
\begin{align*}
\mb{P}[|X|\le\epsilon\sqrt{N}]&=\mb{P}\bigg[\sum_{j=1}^NX_j^2\le\epsilon^2N\bigg] = \mb{P}\bigg[N-\frac{1}{\epsilon^2}\sum_{j=1}^NX_j^2\ge 0\bigg]\\
&\le\mb{E}\exp\bigg(N-\frac{1}{\epsilon^2}\sum_{j=1}^NX_j^2\bigg)\\\
&\le e^N\prod_{j=1}^N\sup_{X_1,\ldots,X_{j-1}}\mb{E}\exp(-X_j^2/\epsilon^2|X_1,\ldots,X_{j-1}).
\end{align*}
We finish by noting that for any realization of $X_1,\ldots,X_{j-1}$, 
\begin{align*}
\mb{E}\exp(-X_j^2/\epsilon^2|X_1,\ldots,X_{j-1}) 
&= \int_0^\infty 2ue^{-u^2}\mb{P}[|X_k| < \epsilon u|X_1,\ldots,X_{j-1}]\,du
\\
&\le\int_0^\infty 2ue^{-u^2}(a_j\epsilon u+b_j)\,du\le a_j\epsilon+b_j.\qedhere
\end{align*}
\end{proof}

We also recall the definition of essential least common denominator (LCD). We use a log-normalized version due to Rudelson (unpublished), which also appears in \cite{Ver14}.
\begin{definition}[LCD]
For $L \ge 1$ and $v \in \mb{S}^{N-1}$, the \emph{least common denominator (LCD)} $D_{L}(v)$ is defined as
\[D_{L}(x) = \on{inf}\left\{\theta > 0: \on{dist}(\theta v, \mb{Z}^{N}) < L\sqrt{\log_{+}(\theta/L)}\right\}.\]
\end{definition}

Finally we will require the following anticoncentration inequality of Miroshnikov and Rogozin \cite{MR80}; this generalizes a well-known inequality of L\'evy-Kolmogorov-Rogozin \cite{Rog61}.

\begin{lemma}[{\cite[Corollary~1]{MR80}}]\label{lem:miroshnikov-rogozin}
Let $\xi_1,\dots, \xi_{N}$ be independent random variables. Then, for any real numbers $r_1,\dots,r_N > 0$ and any real $r \ge \max_{i \in [N]}r_N$, we have
\begin{align*}
    \mc{L}\bigg(\sum_{i=1}^{N}\xi_i, r\bigg) \le C_{\ref{lem:miroshnikov-rogozin}}r \bigg(\sum_{i=1}^{N}\frac{r_i^2(1-\mc{L}(\xi_i,r_i))}{\mc{L}(\xi_i,r_i)^2}\bigg)^{-1/2},
\end{align*}
where $C_{\ref{lem:miroshnikov-rogozin}}>0$ is an absolute constant. 
\end{lemma}

\section{Median Regularized LCD (MRLCD)}\label{sec:MRLCD}

In this section, we introduce the median regularized LCD (MRLCD), which is the notion of arithmetic structure that we will use in the proof of \cref{thm:main}. As opposed to the regularized LCD (RLCD) (introduced in \cite{Ver14}) which guarantees only one arithmetically unstructured projection of the vector, a large MRLCD guarantees many arithmetically unstructured projections of the vector. This simple change allows the MRLCD to piece together various unstructured parts of the vector to obtain significantly better small-ball probability estimates (\cref{prop:anticoncentration}), while at the same time not significantly impacting the size of nets of level sets (\cref{lem:nets-sublevel}). 

\begin{definition}[Median Regularized LCD]\label{def:MRLCD}
For $v \in \on{Incomp}(c_0, c_1)$, $\lambda \in (0,c_{\ref{lem:spread}})$,  and $L \ge 1$, the \emph{median regularized LCD}, denoted $\wh{MD}_{L}(v,\lambda)$, is defined as
\[\wh{MD}_{L}(v,\lambda) = \on{median}\{\{D_L(v_I/\|v_I\|_2): I = \on{Spread}_\lambda^j(v)\text{ for some }j\}\}.\]
Here, the median of an even number of elements is not an average, but instead the value of the upper half. We denote by $I_M(v)$ the set $\on{Spread}_\lambda^j(v)$ achieving the median (arbitrarily chosen from among all such sets), and $\mc{I}_M(v)$ the collection of sets attaining values at least that of the median. We let $\mc{J}_M(v)$ be the collection of sets attaining values at most that of the median.
\end{definition}

We will consider level sets obtained by dyadically chopping the range of the MRLCD. 

\begin{definition}[Level sets of MRLCD]\label{def:level-set} For $\lambda \in (0, c_{\ref{lem:spread}})$, $L \ge 1$, and $D \ge 1$, we define the set
\[S_D = \{v \in \on{Incomp}(c_0, c_1): \wh{MD}_{L}(v,\lambda)\in[D,2D]\}. \]
\end{definition}

\subsection{Nets for level sets of MRLCD}\label{sub:nets}
The main result of this subsection is the following. 
\begin{proposition}\label{lem:nets-sublevel}
Let $c_0, c_1 \in (0,1)$. There exists  $C_{\ref{lem:nets-sublevel}} = C_{\ref{lem:nets-sublevel}}(c_0,c_1) > 0$ for which the following holds. Let $\lambda \in (C_{\ref{lem:nets-sublevel}}/n, c_{\ref{lem:spread}}/2)$ and $L \ge 1$. For every $D \ge 1$, $S_{D}$ has a $\beta$-net $\mathcal{N}$ such that
\[\beta  = \frac{L\sqrt{\log(2D)}}{D}, \quad |\mathcal{N}| \le D^{1/\lambda}\left(\frac{C_{\ref{lem:nets-sublevel}}D}{\sqrt{\log(2D)}}\right)^n\cdot\bigg(\frac{\sqrt{\log(2D)}}{\sqrt{\lambda n}}\bigg)^{c_{\ref{lem:spread}}n/8}.\]
\end{proposition}
\begin{remark}
By changing $C_{\ref{lem:nets-sublevel}}$ by a constant factor we can further assume that $\mc{N}\subseteq S_D$. Also, the $L$ dependence here is not optimal -- one can save a factor of $L^{n(1-c_{\ref{lem:spread}}/8)}$ by being more careful, although this does not affect the overall bounds if $L$ is of constant order as in our application.
\end{remark}

The proof of \cref{lem:nets-sublevel} relies on a bound on the size of nets for level sets of the LCD.

\begin{lemma}[Corollary of Lemma~7.8 in \cite{Ver14}]\label{lem:adjusted-net}
Let $m \in \mb{N}$, $D \ge 1$, and $c\in(0,1)$ be such that $D > c\sqrt{m}\ge 2$. There exists a constant $C$ depending only on $c$ for which the following holds. Let $\chi > 1$, $L \ge 1$, and $\lambda > 0$. Then the set
\[\{x\in\sqrt{\chi\lambda}\mb{B}_2^{m}: c\sqrt{m} < D_L(x/\snorm{x}_2)\le D\}\]
has a $\beta\sqrt{\chi\lambda}$-net $\mc{N}$ such that
\[\beta = \frac{4L\sqrt{\log(2D)}}{D},\quad|\mc{N}|\le\bigg(\frac{CD}{\sqrt{m}}\bigg)^mD^2.\]
\end{lemma}

Now we conclude the result.

\begin{proof}[Proof of \cref{lem:nets-sublevel}]
Let $r = \lceil \lambda n \rceil$ and $k = c_{\ref{def:lambda-spread}}(\lambda)n$, so that $r | k$ by definition.

Now, we pay a factor of $2^n$ in a union bound over possible realizations of $\on{Spread}(v)$, which determines $\on{Spread}_\lambda(v)$ and $\on{Spread}_\lambda^j(v)$ for $1\le j\le k/r$. We pay an additional factor of at most $2^n$ to reveal which sets $\on{Spread}_\lambda^j(v)$ are in $\mc{J}_M(v)$. Let $J\subseteq[k/r]$ be the collection of corresponding indices $j$. We see $|J|\ge k/(2r)$ by definition of median. Write $J = \{j_1,\ldots,j_t\}$ and let $I_i = \on{Spread}_\lambda^{j_i}(v)$.

Note that, given  $I_1,\ldots,I_t$, we know that $D_L(v_{I_i}/\snorm{v_{I_i}}_2)\le 2D$ for all $1\le i\le t$. Moreover, since $I_i \subseteq \on{Spread}_{\lambda}(v)$, it follows that $\snorm{v_{I_i}}_2\le\sqrt{\chi\lambda}$ for some $\chi$ depending only on $c_0$. Further, by \cite[Lemma~6.2]{Ver14}, it follows (again, since $I_i \subseteq \on{Spread}_{\lambda}(v)$) that $D_L(v_{I_i}/\snorm{v_{I_i}}_2) \ge c\sqrt{\lceil \lambda n \rceil}$ for some $c$ depending only on $c_0, c_1$. Hence, by \cref{lem:adjusted-net}, we have a $\beta\sqrt{\chi\lambda}$-net for $v_{I_i}$ where
\[\beta = \frac{2L\sqrt{\log(4D)}}{D}\]
of size at most
\[\bigg(\frac{CD}{\sqrt{\lceil\lambda n\rceil}}\bigg)^{\lceil\lambda n\rceil}D^2.\]
Finally, we take a product of these nets over $1\le i\le t$, along with a standard $\beta$-net of $\mb{B}_2^{I_0}$ (this net has size at most $(1+3/\beta)^{|I_0|}$), where we let $I_0 = [n]\setminus(I_1\cup\cdots\cup I_t)$, to obtain the desired conclusion upon adjusting the value of $\beta$ by standard arguments.
\end{proof}

\subsection{Anticoncentration via MRLCD}\label{sub:anticoncentration}
We derive anticoncentration for a fixed vector with respect to MRLCD; the key idea is to patch together anticoncentration estimates on different segments of the vector through the use of \cref{lem:miroshnikov-rogozin}.

\begin{proposition}[Anticoncentration via the MRLCD]
\label{prop:anticoncentration}
Let $\xi_1,\dots,\xi_{n}$ be i.i.d.~random variables. Suppose that there exist $\epsilon_0, p_0, M_0 > 0$ such that $\mc{L}(\xi_k, \epsilon_0) \le 1-p_0$ and $\mb{E}[|\xi_k|]\le M_0$ for all $k$. Finally, let $c_0, c_1 \in (0,1)$. Then, there exist $C_{\ref{prop:anticoncentration}}$, depending only on $\epsilon_0, p_0, M_0$ and $C'_{\ref{prop:anticoncentration}}$ depending on $\epsilon_0, p_0, M_0, c_0, c_1$ such that the following holds. 

Let $L\ge p_0^{-1/2}$, $\lambda \in (C'_{\ref{prop:anticoncentration}}L^2/n,c_{\ref{lem:spread}})$, $v \in \on{Incomp}(c_0, c_1)$, $J\subseteq\on{Spread}_\lambda(v)$, and $S_J = \sum_{k \in J}v_k \xi_k$. 
Suppose that $J$ is a union of sets in $\mc{I}_M(v)$. Then for every $\epsilon \ge 0$, we have for sufficiently large $n$ (depending on $\epsilon_0, p_0, M_0, c_0, c_1$) that
\[\mathcal{L}(S_J,\epsilon) \le C_{\ref{prop:anticoncentration}}L\left(\frac{\epsilon}{\sqrt{|J|/n}} + \frac{\sqrt{\lambda n/|J|}}{\wh{MD}_{L}(v,\lambda)}\right).\]
\end{proposition}
\begin{remark}
The above proposition should be compared with \cite[Proposition~6.9]{Ver14}, which bounds the L\'evy concentration function in terms of the regularized LCD. The key difference is that the term $\sqrt{|J|/n}$ in the denominator of our bound is replaced by $\sqrt{\lambda}$, which is always smaller. In fact, in the application considered here, $\lambda$ must be chosen to be $O (1/\sqrt{n})$, which makes the above proposition significantly more efficient than the corresponding proposition in \cite{Ver14} for most of the matrix row-vector products (which satisfy $|J|/n = \Theta(1)$). 
\end{remark}
\begin{proof}
Since $v\in\on{Incomp}(c_0,c_1)$, we have $\wh{MD}_L(v,\lambda)\ge c\sqrt{\lceil\lambda n\rceil}$ for some $c$ depending only on $c_0,c_1$ (\cite[Lemma~6.2]{Ver14}).

First, assume that $\epsilon \le 1/(c\sqrt{n})$. Let $r = \lceil \lambda n \rceil$ and $k = c_{\ref{def:lambda-spread}}(\lambda)n$, so that $r | k$ by definition. For $i\in[k/r]$, let
\[S_i = \sum_{k\in\on{Spread}_\lambda^i(v)}v_k\xi_k.\]
Let $I$ be such that $J = \cup_{i\in I}\on{Spread}_\lambda^i(v)$. Since $J$ is a union of sets in $\mc{I}_M(v)$, and $D_L(v_I/\snorm{v_I}_2) \ge \wh{MD}_L(v,\lambda)$ for each $I \in \mc{I}_M(v)$, it follows by standard anticoncentration estimates based on the LCD (see \cite[Proposition~6.9]{Ver14} for the logarithmic version), that there exists an absolute constant $C > 0$ such that
\[\mc{L}(S_i,\epsilon)\le CL\bigg(\frac{\epsilon}{\sqrt{\lambda}}+\frac{1}{\wh{MD}_L(v,\lambda)}\bigg) < \frac{1}{2}\]
for all $i\in I$, where the latter inequality follows from the assumption that $\epsilon \le 1/(c\sqrt{n})$, along with the lower bound on $\lambda$ (by taking $C'_{\ref{prop:anticoncentration}}$ sufficiently large depending on various parameters).

Now note that
\[S_J = \sum_{i\in I}S_i\]
and that the $S_i$ are independent. Also, note that $|I| = |J|/\lceil\lambda n\rceil$. Therefore, by \cref{lem:miroshnikov-rogozin}, we have
\begin{align*}
\mc{L}(S_J,\epsilon)&\le C_{\ref{lem:miroshnikov-rogozin}}\epsilon\bigg(\sum_{i\in I}\frac{\epsilon^2(1-\mc{L}(S_i,\epsilon))}{\mc{L}(S_i,\epsilon)^2}\bigg)^{-1/2}\\
&\le\frac{C_{\ref{lem:miroshnikov-rogozin}}\sqrt{2}}{\sqrt{|I|}}\max_{i\in I}\mc{L}(S_i,\epsilon)\le\frac{2C_{\ref{lem:miroshnikov-rogozin}}}{\sqrt{|J|/n}}\cdot CL\bigg(\epsilon+\frac{\sqrt{\lambda}}{\wh{MD}_L(v,\lambda)}\bigg),
\end{align*}
which proves the desired conclusion for $\epsilon \le 1/c\sqrt{n}$.

Finally, for $\epsilon > 1/(c\sqrt{n})$, we note that any interval of length $2\epsilon$ can be tiled by at most $2\epsilon/\epsilon_0$ intervals of length $2\epsilon_0$, where $\epsilon_0 = 1/(2c\sqrt{n})$. Moreover, for such $\epsilon_0$, we have
\[\epsilon_0 + \frac{\sqrt{\lambda}}{\wh{MD}_L(v,\lambda)} \le 4\epsilon_0.\]
Hence, we have that for all $\epsilon > 1/(c\sqrt{n})$,
\begin{align*}
    \mc{L}(S_J, \epsilon) &\le \frac{2\epsilon}{\epsilon_0}\cdot\mc{L}(S_J, \epsilon_0)\\
    &\le \frac{2\epsilon}{\epsilon_0}\cdot\frac{2C_{\ref{lem:miroshnikov-rogozin}}CL}{\sqrt{|J|/n}}\cdot 4\epsilon_0 \le {16C_{\ref{lem:miroshnikov-rogozin}}CL}\cdot \frac{\epsilon}{\sqrt{\lambda}},
\end{align*}
as desired.
\end{proof}

Next, we derive a small-ball result for (symmetric) matrix-vector products. 

\begin{lemma}\label{lem:small-ball}
Fix $K\ge 1$, $c_0, c_1 \in (0,1)$ and $v \in \on{Incomp}(c_0, c_1)$. There exists $L$ depending only on the sub-Gaussian norm of $\xi$, and  $c_{\ref{lem:small-ball}}, C_{\ref{lem:small-ball}}$ depending on the sub-Gaussian norm of $\xi$ and on $c_0, c_1$ such that the following holds. 

Let $\lambda \in (C_{\ref{lem:small-ball}}/n, c_{\ref{lem:small-ball}})$ and suppose that  $v\in S_D$ (with MRLCD defined with respect to $\lambda, L$). Then, for any $u \in \mb{R}^{n}$, we have 
\[\mb{P}[\snorm{Av-u}_2\le K\beta\sqrt{n}]\le\bigg(\frac{C_{\ref{lem:small-ball}}L^2\sqrt{\log(2D)}}{D}\bigg)^{n-\lceil\lambda n\rceil},\]
where
\[\beta = \frac{L\sqrt{\log(2D)}}{D}.\]
\end{lemma}
\begin{proof}
Fix $u \in \mb{R}^{n}$ and $v \in S_D$. Note that for any permutation matrix $P$, $\snorm{Av-u}_2\le K\beta\sqrt{n}$ occurs if and only if
\[\snorm{(PAP^{-1})Pv-Pu}\le K\beta\sqrt{n}.\]
Furthermore, $PAP^{-1} = PAP^\intercal$ has the same distribution as $A$. Therefore, we will be able to permute the indices of $[n]$ at our convenience (depending on $v$).

In particular, we may assume that $\on{Spread}_\lambda(v) = [c_{\ref{def:lambda-spread}}(\lambda)n]$. Let $A_t = \{k\in[n]: t\lceil\lambda n\rceil < k\le (t+1)\lceil\lambda n\rceil\}$ (defined for $0\le t\le T-1$), where
\[T = \frac{c_{\ref{def:lambda-spread}}(\lambda)n}{\lceil\lambda n\rceil}.\]
We may also assume that $A_{t} \in \mc{I}_M(v)$ for all $t\le \lceil T/2 \rceil - 1$. 
Then, for all $\lceil\lambda n\rceil\le j\le c_{\ref{def:lambda-spread}}(\lambda)n$, the set $J = [j]$ has a subset of at least half the size which satisfies the assumptions of \cref{prop:anticoncentration} (namely, the union of the first $\lfloor j/\lceil\lambda n\rceil\rfloor$ sets $A_t$).

Therefore \cref{prop:anticoncentration} implies that for all $L$ sufficiently large depending on the sub-Gaussian norm of $\xi$, if $j\ge\lceil\lambda n\rceil$ and $\epsilon\ge 0$, then 
\[\mc{L}((Av-u)_j|(Av-u)_{j+1,\ldots,n},\epsilon)\le CC_{\ref{prop:anticoncentration}}L\bigg(\frac{\epsilon}{\sqrt{j/n}}+\frac{\sqrt{\lambda n/j}}{D}\bigg),\]
where $C$ depends only on $c_0,c_1$. Here, we have used that the first $j$ elements of the $j^{th}$ row are independent of rows $j+1,\dots,n$, and that the L\'evy concentration is monotone under removing independent random variables from a sum.

Let $j' = \min(j, c_{\ref{def:lambda-spread}}(\lambda)n)$. Then, by \cref{lem:tensorization}, we deduce
\begin{align*}
\mb{P}[\snorm{Av-u}_2\le K\beta\sqrt{n}]&\le
\mb{P}\left[\sum_{j = \lceil \lambda n \rceil}^{n} (Av-u)_{j}^{2} \le K^{2}\beta^{2} n\right]\\
&\le (CL)^n\prod_{j=\lceil\lambda n\rceil}^n\bigg(\frac{K\beta\sqrt{n/(n-\lceil\lambda n\rceil)}}{\sqrt{j'/n}}+\frac{\sqrt{\lambda n/j'}}{D}\bigg)\\
&\le (CL)^n\prod_{j=\lceil\lambda n\rceil}^n\bigg(\frac{KL\sqrt{n\log(2D)}}{D\sqrt{j'}}\bigg)\\
&\le\bigg(\frac{CL^2\sqrt{\log(2D)}}{D}\bigg)^{n-\lceil\lambda n\rceil},
\end{align*}
where the last inequality uses $\prod_{j=1}^n(n/j)\le e^n$.
\end{proof}

\subsection{Structure theorem}
We will need the following structure theorem, which shows that, except with exponentially small probability, the preimage under $A$ of any fixed vector is highly unstructured. This is our replacement for the key \cite[Theorem~7.1]{Ver14}. As usual, $A$ denotes a random $n\times n$ symmetric matrix with independent $\xi$ entries on and above the diagonal.
\begin{theorem}\label{thm:unstructured-kernel}
Fix $K\ge 1$. Depending on the sub-Gaussian norm of $\xi$, we can choose $L, c, C$ so that the following holds. For all $\lambda\in (C/n,1/\sqrt{n})$, we have for any $u \in \mb{R}^{n}$ that
\[\mb{P}[\exists v\in\mb{S}^{n-1}: (Av\parallel u)\wedge(v\in\on{Comp}(c_0,c_1)\vee\wh{MD}_L(v,\lambda)\le 2^{\lambda n/C})\wedge(\snorm{A}\le K\sqrt{n})]\le 2e^{-cn}.\]
\end{theorem}
\begin{proof}
This is an immediate consequence of \cref{lem:nets-sublevel,lem:small-ball,lem:compressible-singular-vector}. Note that if $Av = tu$ for $t\in\mb{R}$, then $\snorm{A}\le K\sqrt{n}$ implies $\snorm{tu}_2\le K\sqrt{n}$. For compressible vectors $v$, we use \cref{lem:compressible-singular-vector} on a constant amount of target vectors parallel to $u$ so as to cover the full range. For the rest, if the MRLCD is between $D$ and $2D$ (for some $D\le 2^{\lambda n/C}$), we take a net constructed in \cref{lem:adjusted-net} along with a $1/D$-net for $\{tu: \snorm{tu}_2 \le K\sqrt{n}\}$, which adds an additional (unimportant) factor of $KD\sqrt{n}$ to the size of our nets. Since
\begin{align*}
KD\sqrt{n}\cdot D^{1/\lambda}\left(\frac{C_{\ref{lem:nets-sublevel}}D}{\sqrt{\log(2D)}}\right)^n\cdot&\bigg(\frac{\sqrt{\log(2D)}}{\sqrt{\lambda n}}\bigg)^{c_{\ref{lem:spread}}n/8}\times\bigg(\frac{C_{\ref{lem:small-ball}}L^2\sqrt{\log(2D)}}{D}\bigg)^{n-\lceil\lambda n\rceil}\\
&= KD\sqrt{n}\cdot D^{1/\lambda}\bigg(\frac{C'D}{\sqrt{\log(2D)}}\bigg)^{\lceil\lambda n\rceil}\bigg(\frac{\sqrt{\log(2D)}}{\sqrt{\lambda n}}\bigg)^{c_{\ref{lem:spread}}n/8}\\
&\le C'^n\cdot 2^{\lambda^2n^2/C}\cdot (1/C)^{c_{\ref{lem:spread}}n/8}\\
&\le C'^n2^{n/C}(1/C)^{c_{\ref{lem:spread}}n/8},
\end{align*}
the result follows by a union bound upon taking $C$ sufficiently large. We omit the standard details, referring the reader to the proof of \cite[Theorem~7.1]{Ver14} for a more detailed calculation. 
\end{proof}

\section{Median Threshold}\label{sec:median-threshold}
We begin by defining an alternate notion of structure, based on the so-called threshold function (\cref{def:threshold}), which will allow us to use results of Tikhomirov \cite{Tik20} to obtain a stronger bound for the probability of singularity Rademacher random symmetric matrices. We note that, although we have chosen to focus on the Rademacher case, our analysis can be extended to general real discrete distributions using recent results of the authors \cite{JSS20discrete2}. 

For a technical reason that will become clear later, we fix a sufficiently small absolute constant $p\in(0,1/2]$ throughout this section; for the case of Rademacher random variables, which is our focus, one can take $p = 1/10$.

\begin{definition}\label{def:threshold}
For $p\in(0,1)$, $L\ge 1$, and $v\in\mb{S}^{N-1}$, the \emph{threshold} $\mc{T}_{p,L}(v)$ is defined as
\[\mc{T}_{p,L}(v) = \sup\bigg\{t\in(0,1): \mc{L}\bigg(\sum_{i=1}^Nb_i'v_i,t\bigg) > Lt\bigg\},\]
where the $b_i'$ are i.i.d.~random variables distributed as $\on{Ber}(p)-\on{Ber}'(p)$.
\end{definition}
\begin{definition}\label{def:median-threshold}
For $p\in(0,1)$, $v\in\on{Incomp}(c_0,c_1)$, $\lambda\in(0,c_{\ref{lem:spread}})$, and $L\ge 1$, the \emph{median threshold}, denoted $\wh{\mc{T}}_{p,L}(v,\lambda)$, is defined as
\[\wh{\mc{T}}_{p,L}(v,\lambda) = \on{median}\{\{\mc{T}_{p,L}(v_I/\snorm{v_I}_2): I = \on{Spread}_\lambda^j(v)\text{ for some }j\}\}.\]
\end{definition}

\subsection{Threshold of random lattice points}\label{sub:inversion-of-randomness}
We next recall the key technical result of Tikhomirov \cite{Tik20}, which upper bounds the number of vectors with ``large'' threshold within a lattice of appropriate size. The important fact is that the number of such vectors is \emph{superexponentially} small compared to the size of the lattice, which is the key difference with the results coming from the MRLCD. First, we must establish some notation.
\begin{definition}\label{def:admissible-set}
Choose $N,n\ge 1$ and $\delta\in(0,1]$, as well as $K\ge 1$. We say that $\mc{A}\subseteq\mb{Z}^n$ is $(N,n,K,\delta)$\emph{-admissible} if the following hold:
\begin{itemize}
    \item $\mc{A} = A_1\times\cdots\times A_n$, where each $A_i$ is an origin-symmetric subset of $\mb{Z}\cap(-nN,nN)$,
    \item $A_i$ is an integer interval of size at least $2N+1$ for all $i > \delta n$,
    \item $A_i$ is a union of two integer intervals of total size at least $2N$ and $A_i\cap[-N,N] = \emptyset$ for all $i\le\delta n$, and
    \item $|\mc{A}|\le (KN)^n$.
\end{itemize}
\end{definition}

\begin{theorem}[{From~\cite[Corollary~4.3]{Tik20}}]\label{thm:inversion-of-randomness}
Let $\delta,\epsilon\in(0,1]$, $p\in (0,1/2]$, and $K,M\ge 1$. There exist $n_{\ref{thm:inversion-of-randomness}} = n_{\ref{thm:inversion-of-randomness}}(\delta,\epsilon,K,M)\ge 1$ and $L_{\ref{thm:inversion-of-randomness}} = L_{\ref{thm:inversion-of-randomness}}(\delta,\epsilon,K) > 0$ such that the following holds. If $n\ge n_{\ref{thm:inversion-of-randomness}}$, $1\le N\le (1-p+\epsilon)^{-n}$, and $\mc{A}$ is $(N,n,K,\delta)$-admissible, then
\[\bigg|\bigg\{x\in\mc{A}: \mc{L}\bigg(\sum_{i=1}^nb_ix_i,\sqrt{n}\bigg)\ge L_{\ref{thm:inversion-of-randomness}}N^{-1}\bigg\}\bigg|\le\exp(-Mn)|\mc{A}|,\]
where the $b_i$ are i.i.d.~$\on{Ber}(p)$ random variables. Furthermore, $n_{\ref{thm:inversion-of-randomness}} = \exp(C_{\ref{thm:inversion-of-randomness}}(\delta,\epsilon,K)M^2)$ is allowable.
\end{theorem}
\begin{remark}
This is the same as \cite[Corollary~4.3]{Tik20}, except that we have claimed an explicit dependence between $n$ and $M$, namely that one can take $M$ growing as $(\log n)^{1/2}$ (all other parameters fixed). This is an immediate consequence of unraveling the parameter dependencies in \cite[Theorem~4.2]{Tik20}. We give a brief sketch, using the notation of \cite[Theorem~4.2]{Tik20}. In the proof of \cite[Theorem~4.2]{Tik20}, one sets $L = L_{4.5}(2M,p,\delta,\epsilon/2)$, which can be checked to grow exponentially in $M$ by examining the last line of the proof of \cite[Proposition~4.5]{Tik20}. This shows that the parameter $q$ in the proof of \cite[Theorem~4.2]{Tik20} is chosen to be linear in $M$, and hence, the parameter $\wt{\epsilon}$ grows as $M^{-1}$. Next, it is required that $n\ge n_{4.10}(p,\wt{\epsilon},\max(16\wt{R},L),\wt{R},2M)$ and $n \ge n_{4.5}(2M, p, \delta, \epsilon/2)$. The more restrictive condition comes from \cite[Proposition~4.10]{Tik20}, and indeed, an examination of the first few lines of the proof of this proposition reveals that it suffices to have $n$ growing as $\exp(\Theta(M^{2}))$. One also sees that $\eta_{4.2} = \eta_{4.10}(p,\wt{\epsilon},\max(16\wt{R},L),\wt{R},2M)$ decays as $\exp(-\Theta(M^{2}))$. Finally, the deduction of \cite[Corollary~4.3]{Tik20} from \cite[Theorem~4.2]{Tik20} requires $n^{-1/2}\le\eta$, for which $n$ growing as $\exp(\Theta(M^{2}))$ is sufficient in light of the decay of $\eta$ discussed above.
\end{remark}

\subsection{Replacement}\label{sub:replacement}
In order to relate the anticoncentration of a vector with respect to Rademacher random variables to \cref{def:threshold,def:median-threshold}, we will require the following inequality. This is closely related to the replacement trick employed by Kahn, Koml{\'o}s, and Szemer{\'e}di \cite{kahn1995probability} and later by Tao and Vu \cite{tao2007singularity} (although the application here is substantially simpler).
\begin{lemma}\label{lem:replace}
There exists an absolute constant $C_{\ref{lem:replace}}$ for which the following holds. 
Let $v\in \mb{R}^n$ and $r> 0$. Then, for any $0 < p \le (2-\sqrt{2})/4$, 
\[\mc{L}\bigg(\sum_{i=1}^{n}b_iv_i,r\bigg)\le C_{\ref{lem:replace}} \mc{L}\bigg(\sum_{i=1}^{n}b_i'v_i,r\bigg),\]
where $b_i$ are independent Rademacher random variables and $b_i'$ are distributed as $\on{Ber}(p)-\on{Ber}'(p)$.
\end{lemma}
\begin{proof}
Note that by scaling $v$, we may assume without loss of generality that $r = 1$. Let $X = \sum_{i=1}^nb_i'v_i$. By Esseen's inequality and $|\cos t|\le (3+\cos(2t))/4$, we find
\begin{align*}
\mc{L}\bigg(\sum_{i=1}^nb_iv_i,1\bigg)&\le C\int_{-2}^2\prod_{i=1}^n|\cos(v_i\theta)|d\theta\le C\int_{-2}^2\prod_{i=1}^n\bigg(\frac{3}{4}+\frac{1}{4}\cos(2v_i\theta)\bigg)d\theta\\
&\le C\int_{-2}^2 \prod_{i=1}^{n}\mb{E}\exp\bigg(i\theta \cdot  2b_i'v_i\bigg)d\theta\le 2C\int_{\mb{R}}\mbm{1}_{[-2,2]}\ast\mbm{1}_{[-2,2]}(\theta)\mb{E}\exp(i\theta(2X))d\theta\\
&= 4C\mb{E}\bigg(\frac{\sin(4X)}{2X}\bigg)^2\\
&\le 4C\mb{E}\bigg(\frac{\sin(4X)}{2X}\cdot \mbm{1}_{X\in [-1,1]}\bigg)^2 + \sum_{k=1}^{\infty}4C\mb{E}\bigg(\frac{\sin(4X)}{2X}\cdot \mbm{1}_{\pm X\in [2k-1, 2k+1]} \bigg)^2\\
&\le 16C\mc{L}(X,1) + C\sum_{k=1}^{\infty}\frac{\mc{L}(X,1)}{(2k-1)^{2}} \le C'\mc{L}(X,1).
\end{align*}
The third inequality uses $p\le (2-\sqrt{2})/4$, and the penultimate inequality uses $\sin(4x)/(2x) \le 2$ for $x \in [-1,1]$.  
\end{proof}
\subsection{Randomized rounding}\label{sub:randomized-rounding}
We will make use of a slight modification of \cite[Lemma~5.3]{Tik20}, proved using randomized rounding (cf.~\cite{Liv18}). As the proof is identical we omit the details.
\begin{lemma}\label{lem:round}
Let $y = (y_1,\ldots,y_n)\in \mb{R}^n$ be a vector, $\Delta$ be a fixed distribution supported in $[-1,1]^n$, and let $\mu>0$, $\psi\in \mb{R}$ be fixed. There exist absolute constants $c_{\ref{lem:round}}$ and $C_{\ref{lem:round}}$ for which the following holds. 

Suppose that for all $t\ge \sqrt{n}$,
\[\mb{P}\bigg[\bigg|\sum_{i=1}^nb_iy_i-\psi\bigg|\le t\bigg]\le \mu t,\]
where $(b_1,\dots, b_n)$ are independent and distributed as $\Delta$. Then, there exists a vector $y' \in \mb{Z}^{n}$ satisfying
\begin{enumerate}[(R1)]
    \item $\snorm{y-y'}_\infty\le 1$,
    \item $\mb{P}[|\sum_{i=1}^nb_iy_i'-\psi|\le t]\le C_{\ref{lem:round}}\mu t$ for all $t\ge \sqrt{n}$, and
    \item $\mc{L}(\sum_{i=1}^nb_iy_i',\sqrt{n})\ge c_{\ref{lem:round}}\mc{L}(\sum_{i=1}^nb_iy_i,\sqrt{n})$.
\end{enumerate}
\end{lemma}

Next, we prove a version of the above proposition for the case when $\Delta = \on{Ber}(p) - \on{Ber}'(p)$ with $p$ sufficiently small. The main difference is that the left hand side in (R2) above can be replaced by the L\'evy concentration at width $t$; this can be done since for a distribution with non-negative characteristic function, the maximum concentration of given width is essentially obtained around $\psi = 0$.

\begin{lemma}\label{lem:levy-round}
Let $y = (y_1,\ldots,y_n)\in \mb{R}^n$ be a vector, $p\in(0,1)$, and let $\mu>0$, $\psi\in \mb{R}$ be fixed. There exist absolute constants $c_{\ref{lem:levy-round}}$ and $C_{\ref{lem:levy-round}}$ for which the following holds. 

Suppose that for all $t\ge \sqrt{n}$,
\[\mc{L}\bigg(\sum_{i=1}^nb_i'y_i,t\bigg)\le\mu t,\]
where the $b_i'$ are independent and distributed as $\on{Ber}(p)-\on{Ber}'(p)$. Then, there exists a vector $y' = (y_1',\dots,y_n') \in \mb{Z}^{n}$ satisfying
\begin{enumerate}[(R1)]
    \item $\snorm{y-y'}_\infty\le 1$,
    \item $\mc{L}(\sum_{i=1}^nb_i'y_i',t)\le C_{\ref{lem:levy-round}}\mu t$ for all $t\ge \sqrt{n}$, and
    \item $\mc{L}(\sum_{i=1}^nb_i'y_i',\sqrt{n})\ge c_{\ref{lem:levy-round}}\mc{L}(\sum_{i=1}^nb_i'y_i,\sqrt{n})$.
\end{enumerate}
\end{lemma}
\begin{proof}
We apply \cref{lem:round} to the distribution $\Delta = \on{Ber}(p)- \on{Ber}'(p)$ and $\psi = 0$. From (R2),
\[\mb{P}\bigg[\bigg|\sum_{i=1}^nb_i'y_i\bigg|\le t\bigg]\le C_{\ref{lem:round}}\mu t\]
for all $t\ge\sqrt{n}$. Now let $t\ge\sqrt{n}$ and $X = (\sum_{i=1}^nb_i'y_i')/t$, and note that $X$ has nonnegative characteristic function since $b_i'$ does. Thus, for all $\psi\in\mb{R}$,
\begin{align*}
\mb{P}[|X-\psi|\le 1]&=\mb{E}[\mbm{1}_{[-1,1]}(X-\psi)]\le\mb{E}[\mbm{1}_{[-1,1]}\ast\mbm{1}_{[-1,1]}(X-\psi)]\\
&=\int_{\mb{R}}\bigg(\frac{2\sin\theta}{\theta}\bigg)^2\mb{E}\exp(i\theta(X-\psi))d\theta\\
&\le\int_{\mb{R}}\bigg(\frac{2\sin\theta}{\theta}\bigg)^2|\mb{E}\exp(i\theta X)|d\theta\\
&=\int_{\mb{R}}\bigg(\frac{2\sin\theta}{\theta}\bigg)^2\mb{E}\exp(i\theta X)d\theta\\
&= \mb{E}[\mbm{1}_{[-1,1]}\ast\mbm{1}_{[-1,1]}(X)]\le 2\mb{P}[|X|\le 2]\le 4C_{\ref{lem:round}}\mu t.\qedhere
\end{align*}
\end{proof}

\subsection{Threshold structure theorem}\label{sub:threshold-structure}
We now prove the following improved version of \cref{thm:unstructured-kernel}. 
\begin{theorem}\label{thm:threshold-kernel}
Fix $K\ge 1$ and $0 < p\le (2-\sqrt{2})/4$. We can choose $L, c > 0$ and $c' = c'(p)$ so that the following holds for sufficiently large $n$. For all $\lambda\in (n^{-2/3},c(\log n)^{1/4}n^{-1/2})$, we have for any $u \in \mb{R}^{n}$ that
\[\mb{P}[\exists v\in\mb{S}^{n-1}: (Av\parallel u)\wedge(v\in\on{Comp}(c_0,c_1)\vee\wh{\mc{T}}_{p,L}(v,\lambda)\ge 2^{-c'\lambda n})\wedge(\snorm{A}\le K\sqrt{n})]\le 2e^{-cn},\]
where $A$ is a symmetric matrix with entries on and above the diagonal i.i.d.~and distributed as the sum of a Rademacher random variable, and a Gaussian random variable of mean $0$ and variance $n^{-2n}$.
\end{theorem}
\begin{remark}
The Gaussian perturbation of the entries of $A$ is not important here, and will only be used later, where it will be convenient to assume that various sub-matrices of $A$ are invertible almost surely. Moreover, the variance of the Gaussian is chosen sufficiently small so that all anticoncentration claims that we need are essentially unaffected by this perturbation.
\end{remark}
\begin{proof}
As in the proof of \cref{thm:unstructured-kernel}, we can deal with compressible vectors using \cref{lem:compressible-singular-vector}. Therefore, it remains to deal with incompressible vectors with ``large'' median threshold.  

By standard small-ball estimates for incompressible vectors (see \cite[Lemma~5.1]{Tik20}), for $v\in\on{Incomp}(c_0,c_1)$, there is $C_0 = C_0(p,c_0,c_1)$ such that $\wh{\mc{T}}_{p,L}(v,\lambda)\le C_0(\lambda n)^{-1/2}$. We let $r = \lceil\lambda n\rceil$ and $k = c_{\ref{def:lambda-spread}}(\lambda)n$, so that $r|k$ by definition, and let $m = \lfloor k/(2r)\rfloor$.

\textbf{Step 1: Randomized rounding}. We consider the case $\wh{\mc{T}}_{p,L}(v,\lambda)\in[1/T,2/T]$, where $T\in[C_0^{-1}\sqrt{\lambda n},2^{c'\lambda n}]$. Then, by definition, there exist intervals $I_1,\ldots,I_m$ of the form $\on{Spread}_\lambda^j(v)$ with
\[{\mc{T}}_{p,L}(v_{I_i}/\snorm{v_{I_i}}_2)\le 2/T\]
for all $i\in[m]$. 

Let $D = C_1\sqrt{n}T$, where $C_1 = C_1(p,c_0,c_1)\ge 1$ will be an integer chosen later. Let $y = Dv$. By the definition of the threshold, for all $t\ge\sqrt{\lceil\lambda n\rceil}$ we have
\begin{align*}
\mc{L}\bigg(\sum_{j\in I_i}b_j'y_j,t\bigg) &= \mc{L}\bigg(\sum_{j\in I_i}b_j'v_j,\frac{t}{D}\bigg) = \mc{L}\bigg(\sum_{j\in I_i}b_j'\frac{v_j}{\snorm{v_{I_i}}_{2}},\frac{t}{D\snorm{v_{I_i}}_{2} }\bigg)\\
&\le\mc{L}\bigg(\sum_{j\in I_i}b_j'\frac{v_j}{\snorm{v_{I_i}}_2},\frac{\sqrt{2}t}{c_1C_1\sqrt{\lceil\lambda n\rceil}T}\bigg)\le\frac{L}{T}\cdot\frac{2t}{\sqrt{\lceil\lambda n\rceil}},
\end{align*}
as long as we chose $C_1 > \sqrt{2}/c_1$, where the $b_i'$ are independent random variables distributed as $\on{Ber}(p)-\on{Ber}'(p)$. Applying \cref{lem:levy-round} to the $\lceil\lambda n\rceil$-dimension vector $y_{I_i}$, we see that there is $y_{I_i}'\in\mb{Z}^{\lceil\lambda n\rceil}$ satisfying the conclusions of \cref{lem:levy-round} (with $n$ replaced by $\lceil\lambda n\rceil$). In particular, by (R3), we see that
\begin{align}
\label{eq:lower-bound-levy}
    \mc{L}\bigg(\sum_{j\in I_i}b_j'y_j', \sqrt{\lceil\lambda n\rceil}\bigg) 
    &\ge c_{\ref{lem:levy-round}}\mc{L}\bigg(\sum_{j\in I_i}b_j'v_j, \sqrt{\lambda}/(C_1T)\bigg) \nonumber\\
    & \ge c_{\ref{lem:levy-round}}\mc{L}\bigg(\sum_{j\in I_i}b_j'\frac{v_j}{\snorm{v_{I_i}}_{2}}, \frac{2\sqrt{\lambda}}{C_1T\cdot \sqrt{\lambda/c_0}}\bigg) \nonumber\\
    &\ge C_1^{-1}\sqrt{c_0}\cdot c_{\ref{lem:levy-round}}\mc{L}\bigg(\sum_{j\in I_i}b_j'\frac{v_j}{\snorm{v_{I_i}}_2}, 2/T\bigg) \nonumber\\
    &\ge C_1^{-1}\sqrt{c_0}\cdot c_{\ref{lem:levy-round}}\cdot 2LT^{-1}.
\end{align}
Let $I_0 = [n]\setminus (I_1 \cup \dots \cup I_m)$. Then, by approximating each coordinate of $y_{I_0}$ by the nearest integer, and combining with the above integer approximations of $y_{I_1},\dots, y_{I_m}$, we obtain an integer vector $y'\in\mb{Z}^n$.

\textbf{Step 2: Size of nets of level sets.} We now estimate the number of possible realizations $y'$. This is the analogue of \cref{lem:nets-sublevel} in the present context. By paying an overall factor of at most $6^n$, we may fix $\on{Spread}(v)$ (hence all the $\on{Spread}_\lambda^j(v)$), as well as which $\on{Spread}_\lambda^j(v)$ are in $\mc{I}_M(v)$ and $\mc{J}_M(v)$. As above, let us denote the intervals $\on{Spread}_\lambda^j(v)$ in $\mc{I}_M(v)$ by $I_1,\dots, I_m$, and let $I_0 = [n]\setminus (I_1\cup \dots \cup I_m)$. 

First, note that the number of choices for $y'_{I_0}$ is at most $(CD/\sqrt{n})^{|I_0|}$ for an absolute constant $C$ -- this follows since $y'_{I_0}$ is an integer point in a ball of radius $D \ge \sqrt{n} \ge \sqrt{|I_0|}$ (provided that $C_1$ is chosen sufficiently large), at which point, we can use a standard volumetric estimate for the number of integer points in $\mathbb{R}^{I_0}$ in a ball of radius $R \ge \sqrt{|I_0|}$, together with the bound $|I_0| \ge n/2$.  

Next, we fix $i \in [m]$, and bound the number of choices for $y'_{I_i}$. Note that for any $r \ge 0$,
\[\mc{L}\left(\sum_{j \in I_i}b_j y'_j, r\right) \ge \mc{L}\left(\sum_{j \in I_i}(b_j - \wt{b}_j) y'_j, r\right),\]
where $b_i, \wt{b}_i$ are independent copies of $\on{Ber}(p)$. Since $b'_j$ is distributed as $b_j - \wt{b}_j$, it follows from \cref{eq:lower-bound-levy} that
\[\mc{L}\bigg(\sum_{j\in I_i}b_jy_j',\sqrt{\lceil\lambda n\rceil}\bigg)\ge c_2LN^{-1},\]
where $b_j$ are i.i.d.~$\on{Ber}(p)$ random variables. From the definition of $\on{Spread}_{\lambda}(v)$ and (R1), we see that $y_{I_i}'$ lies within a $(D/(C_2\sqrt{n}),\lceil\lambda n\rceil,K',1)$-admissible set $\mc{A}$ for $C_2$ and $K'$ sufficiently large depending on $c_0, c_1$. Then, for $L$ sufficiently large depending on $c_0, c_1, p$, by \cref{thm:inversion-of-randomness} (noting that $D$ is bounded by $n2^{c'\lambda n}$ for all sufficiently large $n$, and that we can take $c'$ to be sufficiently small depending on $p$), we deduce that the number of potential $y_{I_i}'\in\mb{Z}^{I_i}$ is bounded by
\[\exp(-M|I_i|)(CD/\sqrt{n})^{|I_i|},\]
where $C$ depends on $c_0, c_1$, and $M$ grows as $\sqrt{\log\lceil\lambda n\rceil}$, hence as $\sqrt{\log n}$. Explicitly, we can pick $M\ge c_3\sqrt{\log n}$ for some small $c_3 > 0$ depending only on $c_0, c_1, p$. Multiplying the total number of possibilities for $y'_0,y'_1,\dots, y'_m$, we see that the total number of possibilities for $y' \in \mb{Z}^{n}$ is at most
\[\exp(-c_{\ref{lem:spread}}Mn/8)(CD/\sqrt{n})^n,\]
for $C$ depending on $c_0, c_1$ and $M \ge c_3 \sqrt{\log{n}}$ with $c_3$ depending on $c_0, c_1, p$.

\textbf{Step 3: Small-ball probability for net points.} Fix $y' \in \mb{Z}^{n}$ resulting from the randomized rounding process and $u \in \mathbb{R}^{n}$. Our goal is to bound $\mb{P}[\snorm{Ay'-u}_{2} \le Kn]$. As in the proof of \cref{lem:small-ball}, we can without loss of generality permute the coordinates of $y'$ so that $I_1,\ldots,I_m$ are the first $m$ blocks of size $\lceil\lambda n\rceil$ within $[n]$. Then, for all $\lceil\lambda n\rceil\le j\le c_{\ref{def:lambda-spread}}(\lambda)n$, we have for all $\epsilon \ge 0$ that
\[\mc{L}((Ay'-u)_j|(Ay'-u)_{j+1,\ldots,n},D\epsilon)\le CL\bigg(\frac{\epsilon}{\sqrt{j/n}}+\frac{\sqrt{\lambda n/j}}{{T}}\bigg),\]
where $C$ is an absolute constant. 
To deduce this, we use that the first $j$ elements of row $j$ are independent of rows $j+1,\dots, n$, then use \cref{lem:replace} to replace the Rademacher entries of $A$ (plus the small Gaussian perturbation, which has variance so small that it can be disregarded) by $\on{Ber}(p)-\on{Ber}'(p)$, and finally use \cref{lem:miroshnikov-rogozin} (as in the proof of \cref{prop:anticoncentration}) to stitch together the L\'evy concentration properties of each $y_{I_i}'$ (guaranteed by (R3) of \cref{lem:levy-round}). Combining this with \cref{lem:tensorization}, we see that for $y', u$ as above,
\[\mb{P}[\snorm{Ay'-u}_{2} \le Kn] 
\le \bigg(\frac{C''L\sqrt{n}}{D}\bigg)^{n-\lceil\lambda n\rceil},\]
where $C''$ depends only on $c_0, c_1, p$.

\textbf{Step 4: Union bound.} On the event $\snorm{A}\le K\sqrt{n}$, $Av = tu$ with $\snorm{tu}_{2} \le K\sqrt{n}$. By splitting the range $\{tu\}$ into $(4D/\sqrt{n})^{2}$ intervals, and rounding $v$ as in Step 2, we see that the probability of the event in question is bounded above by 
\[\exp(-c_{\ref{lem:spread}}Mn/8)\bigg(\frac{CD}{\sqrt{n}}\bigg)^{n+2}\sup_{y',u}\mb{P}[\snorm{Ay'-u}_2\le Kn],\]
where the supremum is over $u \in \mb{R}^{n}$ and $y' \in \mb{Z}^{n}$ such that each $y_{I_i}'$ for $i\in[m]$ satisfies the conclusions of \cref{lem:levy-round} (with $n$ replaced by $\lceil\lambda n\rceil$). Controlling the final factor by Step 3, we see that the probability is bounded above by
\[\exp(-c_{\ref{lem:spread}}Mn/8)C^nD^{\lceil\lambda n\rceil+2},\]
where $C$ depends on $c_0, c_1, p$.
Finally, since $D\le 2^{2c'\lambda n}$ for all $n$ sufficiently large (depending on $c_0, c_1, p$), we obtain an overall upper bound of
\[\exp(-c_{\ref{lem:spread}}Mn/8 + n\log C + 2c'\lambda^2n^2 + 6c'\lambda n).\]
Since $M \ge c_3 \sqrt{\log{n}}$, for $c_3$ depending on $c_0, c_1, p$, the desired result follows by choosing 
$\lambda < c(\log n)^{1/4}n^{-1/2}$ for $c$ sufficiently small depending on $c_0, c_1, p$, so that the quantity above is bounded by  $\exp(-\Omega(n(\log n)^{1/2}))$.
\end{proof}

\section{Proof of \texorpdfstring{\cref{thm:main}}{Theorem 1.1}}\label{sec:deriving}

In this section (along with \cref{app:vershynin}), we complete the proof of \cref{thm:main} by closely following \cite{Ver14} with appropriate modifications. Since the smallest singular value is a continuous function of the entries of the matrix, by perturbing each entry of the random matrix by a Gaussian variable with arbitrarily small variance, one may assume that $\xi$ is absolutely continuous with respect to the Lebesgue measure; in particular, one may freely assume that various square matrices whose entries are independent copies of $\xi$ are invertible.

\subsection{Quadratic  small-ball probabilities}\label{sub:reduction}
To prove \cref{thm:main,thm:main-2}, we need the following small-ball inequalities for quadratic forms. The derivation is almost identical to the approach in \cite[Theorem~8.1]{Ver14}, with improvements coming from \cref{thm:unstructured-kernel} and \cref{prop:anticoncentration} (respectively \cref{thm:threshold-kernel}). We include details in the appendix for the reader's convenience.
\begin{theorem}\label{thm:quadratic-form-small-ball}
Let $A$ be an $n\times n$ symmetric random matrix whose independent entries are identical copies of a sub-Gaussian random variable $\xi$ with variance $1$. Suppose $X$ is a random vector (independent of $A$) whose entries are independent copies of $\xi$. Then, for every $\epsilon\ge 0$ and $u\in\mb{R}$, we have
\[\mb{P}\bigg[\frac{|\sang{A^{-1}X,X}-u|}{\sqrt{1+\snorm{A^{-1}X}_2^2}}\le\epsilon\wedge\snorm{A}\le K\sqrt{n}\bigg]\le C_{\ref{thm:quadratic-form-small-ball}}\epsilon^{1/8}+2\exp(-c_{\ref{thm:quadratic-form-small-ball}}n^{1/2}).\]
\end{theorem}
We similarly derive the following strengthening for Rademacher entries.
\begin{theorem}\label{thm:quadratic-form-small-ball-rademacher}
Let $A$ be an $n\times n$ symmetric random matrix whose independent entries are distributed as the sum of a Rademacher random variable and a centered Gaussian with variance $n^{-2n}$. Suppose $X$ is a random vector (independent of $A$) whose entries are independent Rademachers. Then, for all sufficiently large $n$, and for every $\epsilon\ge 0$ and $u\in\mb{R}$, we have
\[\mb{P}\bigg[\frac{|\sang{A^{-1}X,X}-u|}{\sqrt{1+\snorm{A^{-1}X}_2^2}}\le\epsilon\wedge\snorm{A}\le K\sqrt{n}\bigg]\le C_{\ref{thm:quadratic-form-small-ball}}\epsilon^{1/8}+2\exp(-c_{\ref{thm:quadratic-form-small-ball}}n^{1/2}(\log n)^{1/4}).\]
\end{theorem}

\subsection{Putting it together}\label{sub:conclusion}
Given the above, the proofs of \cref{thm:main,thm:main-2} follows from a modification (due to Vershynin) of the invertibility-via-distance paradigm due to Rudelson and Vershynin. We reproduce the details from \cite{Ver14} for the reader's convenience
\begin{proof}[Proof of \cref{thm:main,thm:main-2}]
Fix $c_0, c_1,c \in (0,1)$, as guaranteed by \cref{lem:compressible-singular-vector}. We can clearly assume that $\epsilon \le c$.   
Then, by the union bound and \cref{lem:compressible-singular-vector}, we have
\begin{align*}
\mb{P}[s_n(A)\le\epsilon/\sqrt{n}]&\le\mb{P}[\exists v\in\on{Comp}(c_0,c_1): \snorm{Av}_2\le c\sqrt{n}] + \mb{P}[\exists v\in\on{Incomp}(c_0,c_1): \snorm{Av}_2\le\epsilon/\sqrt{n}]\\
&\le 2\exp(-cn) + \mb{P}[\exists v\in\on{Incomp}(c_0,c_1): \snorm{Av}_2\le\epsilon/\sqrt{n}].
\end{align*}
Let $A_1,\dots, A_n$ denote the rows of $A$, and note that, by symmetry,
\[Av = \sum_{i=1}^nv_iA_i^T.\]
In particular, 
\[\snorm{Av}_2 \ge |v_i|\on{dist}(A_i, H_i),\]
where $H_i$ is the span of the rows $A_j$ for $j\neq i$. Since $|v_i| \ge c_1/2\sqrt{n}$ for all $i \in \on{Spread}(v)$, it follows that if $\snorm{Av}_2 \le \epsilon/\sqrt{n}$ for some $v \in \on{Incomp}(c_0, c_1)$, then we must necessarily have
\[\on{dist}(A_i,H_i)\le\frac{\epsilon\sqrt{2}}{c_1}\]
for at least $c_{\ref{lem:spread}}n$ indices $i \in [n]$. Thus, we see that the probability that $s_n(A) \le \epsilon/\sqrt{n}$ is at most
\[2\exp(-cn) + \frac{1}{c_{\ref{lem:spread}}n}\sum_{i=1}^n\mb{P}\bigg[\on{dist}(A_i,H_i)\le\frac{\epsilon\sqrt{2}}{c_1}\bigg].\]
Therefore, for \cref{thm:main} it suffices to show that
\[\mb{P}[\on{dist}(A_1,H_1)\le\epsilon]\le C\epsilon^{1/8}+2\exp(-cn^{1/2}).\]
A direct computation (\cite[Proposition~5.1]{Ver14}) shows that 
\[\on{dist}(A_1,H_1) = \frac{|\sang{(A')^{-1}X,X}-a_{11}|}{\sqrt{1+\snorm{(A')^{-1}X}_2^2}},\]
where $A'$ is the bottom right $(n-1)\times(n-1)$ block of $A$, and $X$ is the first column of $A$ with the top element removed. At this point, we can apply \cref{thm:quadratic-form-small-ball} to conclude. If $A$ has Rademacher entries, by continuity we  can transfer the singular value estimate to the model where the distribution is perturbed by a centered Gaussian with sufficiently small variance, at which point, an application of  \cref{thm:quadratic-form-small-ball-rademacher} allows us to conclude. 
\end{proof}

\bibliographystyle{amsplain0.bst}
\bibliography{main.bib}

\providecommand{\bysame}{\leavevmode\hbox to3em{\hrulefill}\thinspace}
\providecommand{\MR}{\relax\ifhmode\unskip\space\fi MR }
\providecommand{\MRhref}[2]{%
  \href{http://www.ams.org/mathscinet-getitem?mr=#1}{#2}
}
\providecommand{\href}[2]{#2}
\begin{thebibliography}{10}

\bibitem{balogh2018method}
J\'{o}zsef Balogh, Robert Morris, and Wojciech Samotij, \emph{The method of
  hypergraph containers}, Proceedings of the {I}nternational {C}ongress of
  {M}athematicians---{R}io de {J}aneiro 2018. {V}ol. {IV}. {I}nvited lectures,
  World Sci. Publ., Hackensack, NJ, 2018, pp.~3059--3092.

\bibitem{bourgain2010singularity}
Jean Bourgain, Van~H. Vu, and Philip~Matchett Wood, \emph{On the singularity
  probability of discrete random matrices}, Journal of Functional Analysis
  \textbf{258} (2010), 559--603.

\bibitem{campos2019singularity}
Marcelo Campos, Let{\'\i}cia Mattos, Robert Morris, and Natasha Morrison,
  \emph{On the singularity of random symmetric matrices}, Duke Mathematical
  Journal (2020), to appear.

\bibitem{CTV06}
Kevin~P. Costello, Terence Tao, and Van Vu, \emph{Random symmetric matrices are
  almost surely nonsingular}, Duke Math. J. \textbf{135} (2006), 395--413.

\bibitem{ferber2019singularity}
Asaf Ferber and Vishesh Jain, \emph{Singularity of random symmetric
  matrices—a combinatorial approach to improved bounds}, Forum of
  Mathematics, Sigma, vol.~7, Cambridge University Press, 2019.

\bibitem{FJLS2018}
Asaf Ferber, Vishesh Jain, Kyle Luh, and Wojciech Samotij, \emph{On the
  counting problem in inverse {L}ittlewood--{O}fford theory}, arXiv:1904.10425.

\bibitem{JSS20discrete2}
Vishesh Jain, Ashwin Sah, and Mehtaab Sawhney, \emph{Singularity of discrete
  random matrices {II}}, arXiv:2010.06554.

\bibitem{JSS20}
Vishesh Jain, Ashwin Sah, and Mehtaab Sawhney, \emph{The smallest singular
  value of dense random regular digraphs}, arXiv:2008.04755.

\bibitem{kahn1995probability}
Jeff Kahn, J{\'a}nos Koml{\'o}s, and Endre Szemer{\'e}di, \emph{On the
  probability that a random $\pm$1-matrix is singular}, Journal of the American
  Mathematical Society \textbf{8} (1995), 223--240.

\bibitem{komlos1967determinant}
J{\'a}nos Koml{\'o}s, \emph{On determinant of (0, 1) matrices}, Studia Science
  Mathematics Hungarica \textbf{2} (1967), 7--21.

\bibitem{Liv18}
Galyna~V Livshyts, \emph{The smallest singular value of heavy-tailed not
  necessarily iid random matrices via random rounding}, arXiv:1811.07038.

\bibitem{livshyts2019smallest}
Galyna~V Livshyts, Konstantin Tikhomirov, and Roman Vershynin, \emph{The
  smallest singular value of inhomogeneous square random matrices},
  arXiv:1909.04219.

\bibitem{lopatto2019tail}
Patrick Lopatto and Kyle Luh, \emph{Tail bounds for gaps between eigenvalues of
  sparse random matrices}, arXiv:1901.05948.

\bibitem{luh2018sparse}
Kyle Luh and Van Vu, \emph{Sparse random matrices have simple spectrum},
  arXiv:1802.03662.

\bibitem{MR80}
A.~L. Miro\v{s}nikov and B.~A. Rogozin, \emph{Inequalities for concentration
  functions}, Teor. Veroyatnost. i Primenen. \textbf{25} (1980), 178--183.

\bibitem{nguyen2017random}
Hoi Nguyen, Terence Tao, and Van Vu, \emph{Random matrices: tail bounds for
  gaps between eigenvalues}, Probability Theory and Related Fields \textbf{167}
  (2017), 777--816.

\bibitem{Ngu12}
Hoi~H. Nguyen, \emph{Inverse {L}ittlewood-{O}fford problems and the singularity
  of random symmetric matrices}, Duke Math. J. \textbf{161} (2012), 545--586.

\bibitem{nguyen2012least}
Hoi~H. Nguyen, \emph{On the least singular value of random symmetric matrices},
  Electron. J. Probab. \textbf{17} (2012), no. 53, 19.

\bibitem{o2016conjecture}
Sean O'Rourke and Behrouz Touri, \emph{On a conjecture of {G}odsil concerning
  controllable random graphs}, SIAM Journal on Control and Optimization
  \textbf{54} (2016), 3347--3378.

\bibitem{Rog61}
B.~A. Rogozin, \emph{On the increase of dispersion of sums of independent
  random variables}, Teor. Verojatnost. i Primenen \textbf{6} (1961), 106--108.

\bibitem{rudelson2008littlewood}
Mark Rudelson and Roman Vershynin, \emph{The {L}ittlewood--{O}fford problem and
  invertibility of random matrices}, Advances in Mathematics \textbf{218}
  (2008), 600--633.

\bibitem{tao2006random}
Terence Tao and Van Vu, \emph{On random$\pm$1 matrices: singularity and
  determinant}, Random Structures \& Algorithms \textbf{28} (2006), 1--23.

\bibitem{tao2007singularity}
Terence Tao and Van~H. Vu, \emph{On the singularity probability of random
  {B}ernoulli matrices}, Journal of the American Mathematical Society
  \textbf{20} (2007), 603--628.

\bibitem{Tik20}
Konstantin Tikhomirov, \emph{Singularity of random {B}ernoulli matrices}, Ann.
  of Math. (2) \textbf{191} (2020), 593--634.

\bibitem{Ver14}
Roman Vershynin, \emph{Invertibility of symmetric random matrices}, Random
  Structures Algorithms \textbf{44} (2014), 135--182.

\bibitem{wei2017investigate}
Feng Wei, \emph{Investigate invertibility of sparse symmetric matrix},
  arXiv:1712.04341.

\end{thebibliography}

\appendix
\section{Quadratic small-ball probabilities }\label{app:vershynin}
The purpose of this appendix is to prove \cref{thm:quadratic-form-small-ball} for completeness. We will also briefly note the necessary modifications to deduce \cref{thm:quadratic-form-small-ball-rademacher}. 

We essentially replicate the argument in \cite[Section~8]{Ver14} with the obvious modifications. In brief, our improved anticoncentration estimate \cref{prop:anticoncentration} will allow us to replace the $\epsilon^{1/8+\eta}$ dependence in \cite{Ver14} with $\epsilon^{1/8}$, and the improved range of arithmetic structure derived in \cref{thm:unstructured-kernel} will allow us to achieve an error term of $\exp(-\Omega(\sqrt{n}))$. For the sake of simplicity, we define the event
\[\mc{E}_K := \{\snorm{A}\le K\sqrt{n}\}.\]

\begin{proposition}[Analogue of {\cite[Proposition~8.2]{Ver14}}]\label{prop:denominator-size}
Let $A$ be a symmetric random matrix whose independent entries are identical copies of a sub-Gaussian random variable with mean $0$ and variance $1$. Let $X$ be a random vector (independent of $A$) whose coordinates are i.i.d.~copies of $\xi$. There exist constants $C, c > 0$ depending only on the sub-Gaussian moment of $\xi$ for which the following holds. 

For $\lambda\in(C/n,1/\sqrt{n})$, $A$ satisfies the following with probability at least $1-2e^{-cn}$:  if $\mc{E}_K$ holds, then for every $\epsilon > 0$:
\begin{itemize}
    \item $\snorm{A^{-1}X}_2\ge c$ with probability at least $1-e^{-cn}$ in the randomness of $X$.
    \item $\snorm{A^{-1}X}_2\le\epsilon^{-1/2}\snorm{A^{-1}}_{\on{HS}}$ with probability at least $1-\epsilon$ in the randomness of $X$.
    \item $\snorm{A^{-1}X}_2\ge\epsilon\snorm{A^{-1}}_{\on{HS}}$ with probability at least $1-C\epsilon-2e^{-c\lambda n}$ in the randomness of $X$.
 \end{itemize}
\end{proposition}
\begin{proof}
The first two parts have the same proof as in \cite[Proposition~8.2]{Ver14}. The last part also has essentially the same proof, except that we use \cref{prop:anticoncentration} in place of \cite[Proposition~6.9]{Ver14} and use \cref{thm:unstructured-kernel} in place of \cite[Theorem~7.1]{Ver14}.
\end{proof}
\begin{remark}
For \cref{thm:quadratic-form-small-ball-rademacher}, we note that if $A$ has entries which are Rademacher plus a centered Gaussian with sufficiently small variance, we can prove the same statement with $n^{-2/3} < \lambda < c(\log n)^{1/4}n^{-1/2}$. We use \cref{thm:threshold-kernel} instead of \cref{thm:unstructured-kernel} and the analogue of \cref{prop:anticoncentration} for the threshold. The remaining part of the proof is exactly the same.
\end{remark}

Next, we require the following decoupling lemma from \cite{Ver14}; this use of decoupling to establish singularity for symmetric random matrices originates in work of Costello, Tao, and Vu \cite{CTV06}, and has been used in essentially all follow-up works.

\begin{lemma}[{\cite[Lemma~8.4]{Ver14}}]\label{lem:decoupling-quadratic-forms}
Let $G$ be an arbitrary symmetric $n\times n$ matrix, and let $X,X'$ be independent samples of a random vector in $\mb{R}^n$ with independent coordinates. Let $J\subseteq[n]$. Then, for every $\epsilon\ge 0$ we have
\[\mc{L}(\sang{GX,X},\epsilon)^2\le\mb{P}_{X,X'}[|\sang{GP_{J^c}(X-X'),P_JX}-v|\le\epsilon]\]
for some random variable $v$ determined by $G|_{J^c\times J^c}$ and $P_{J^c}X, P_{J^c}X'$.
\end{lemma}


We can now prove \cref{thm:quadratic-form-small-ball}; we refer the reader to \cite{Ver14} for a more detailed exposition. 

\begin{proof}[Proof of \cref{thm:quadratic-form-small-ball}]
We randomly choose $J\subseteq[n]$ by sampling elements independently with probability $1-c_{\ref{lem:spread}}/2$. We trivially see by the Chernoff bound that if $\mc{E}_J = \{|J^c|\le c_{\ref{lem:spread}}n\}$, then
\[\mb{P}[\mc{E}_J]\ge 1-2e^{-cn}.\]
For $J$ satisfying $\mc{E}_J$, let us assign the set $\on{Spread}(v)$ for $v\in\on{Incomp}(c_0,c_1)$ in a way such that $\on{Spread}(x)\subseteq |J|$. We can do this since, in \cref{lem:spread}, we chose $c_{\ref{lem:spread}}$ so as to have at least $2c_{\ref{lem:spread}}n$ spread coordinates. We will then use this assignment to obtain the median regularized LCDs that are used.

Next, consider the event $\mc{E}_D$ given by
\[\epsilon_0^{1/2}\sqrt{1+\snorm{A^{-1}X}_2^2}\le\snorm{A^{-1}}_{\on{HS}}\le\frac{1}{\epsilon_0}\snorm{A^{-1}P_{J^c}(X-X')}_2.\]
Applying \cref{prop:denominator-size} to $X$ and $Y_i = \delta_i(X_i-X_i')$, where $\delta_i$ is the indicator of $i\in J^c$, and adjusting constants appropriately, we find that
\[\mb{P}_{J,A,X,X'}[\mc{E}_D\vee\mc{E}_K^c]\ge 1 - C\epsilon_0-2e^{-c\lambda n}-2e^{-cn},\]
where the constants depend only on the sub-Gaussian norm of $\xi$. Now define
\[x_0 = \frac{A^{-1}P_{J^c}(X-X')}{\snorm{A^{-1}P_{J^c}(X-X')}_2},\]
which is a random vector. If the denominator is $0$, we can use an arbitrary fixed vector. Let $\mc{E}_U$ be the event (analogue of \cite[Equation~8.11]{Ver14}) that 
\[x_0\in\on{Incomp}(c_0,c_1),\qquad \wh{MD}_L(x_0,\lambda)\ge 2^{\lambda n/C},\]
where we choose $L$ as in \cref{thm:unstructured-kernel}.

Now condition on some $J$ satisfying $\mc{E}_J$ and some $X,X'$. By \cref{thm:unstructured-kernel}, we deduce that
\[\mb{P}_A[\mc{E}_U\vee\mc{E}_K^c|X,X',J]\ge 1-2e^{-cn}.\]
Thus (analogue of \cite[Equation~8.12]{Ver14})
\[\mb{P}_{J,A,X,X'}[(\mc{E}_J\wedge\mc{E}_D\wedge\mc{E}_U)\vee\mc{E}_K^c]\ge 1-p_0,\]
where $p_0 = C\max(\epsilon_0,2^{-\lambda n/C})$. Hence, there is a realization of $J$ such that $\mc{E}_J$ holds and
\[\mb{P}_{A,X,X'}[(\mc{E}_D\wedge\mc{E}_U)\vee\mc{E}_K^c]\ge 1-p_0.\]
We fix this choice of $J$ for the remainder of the proof. Now let $\mc{E}_A$ be the event, dependent only on $A$, that simultaneously $\mc{E}_K$ and
\[\mb{P}_{X,X'}[\mc{E}_D\wedge\mc{E}_U|A]\ge 1-p_0^{1/2}.\]
By Fubini's theorem, Markov's inequality, and the fact that $\mc{E}_K$ depends only on $A$, we see from the above that (analogue of \cite[Equation~8.13]{Ver14})
\[\mb{P}_A[\mc{E}_A\vee\mc{E}_K^c]\ge 1-p_0^{1/2}.\]

Now, if $\mc{E}$ is the desired event
\[\frac{|\sang{A^{-1}X,X}-u|}{\sqrt{1+\snorm{A^{-1}X}_2^2}}\le\epsilon,\]
then
\begin{align*}
\mb{P}_{A,X}[\mc{E}]\le \mb{P}[\mc{E}_K^c] + p_0^{1/2} + \sup_{A\in\mc{E}_A}\mb{P}_X[\mc{E}|A].
\end{align*}
Fix some $A\in\mc{E}_A$ for the remainder of the proof. We need to bound
\[\mb{P}_X[\mc{E}|A]\le\mb{P}_{X,X'}[\mc{E}\wedge\mc{E}_D|A] + p_0^{1/2}.\]
Using $\mc{E}_D$ along with $\mc{E}$, we see that
\[\mb{P}_{X,X'}[\mc{E}\wedge\mc{E}_D|A]\le\mb{P}_{X,X'}[|\sang{A^{-1}X,X}-u|\le\epsilon\epsilon_0^{-1/2}\snorm{A^{-1}}_{\on{HS}}|A] =: p_1.\]
Then by \cref{lem:decoupling-quadratic-forms}, we find that this satisfies
\[p_1^2\le\mb{P}_{X,X'}[|\sang{A^{-1}P_{J^c}(X-X'),P_JX}-v|\le\epsilon\epsilon_0^{-1/2}\snorm{A^{-1}}_{\on{HS}}|A],\]
where $v = v(A^{-1},P_{J^c}X,P_{J^c}X')$ is some random variable depending on these parameters only. This last probability is at most
\[p_0^{1/2} + \mb{P}_{X,X'}[|\sang{A^{-1}P_{J^c}(X-X'),P_JX}-v|\le\epsilon\epsilon_0^{-1/2}\snorm{A^{-1}}_{\on{HS}}\wedge\mc{E}_D\wedge\mc{E}_U|A].\]
Now by using $\mc{E}_D$ again, and dividing the inequality in question by $\snorm{A^{-1}P_{J^c}(X-X')}_2$, we see that (analogue of \cite[Equation~8.15]{Ver14})
\[p_1^2\le p_0^{1/2} + \mb{P}_{X,X'}[|\sang{x_0,P_JX}-w|\le\epsilon_0^{-3/2}\epsilon\wedge\mc{E}_U|A].\]
Finally, we can apply \cref{prop:anticoncentration} to this random variable. Note that $x_0,w$ do not depend on $P_JX$. Also, we know from $\mc{E}_U$ that $\wh{MD}_L(x_0,\lambda)\ge 2^{\lambda n/C}$. It suffices to check that $\on{Spread}_\lambda(x_0)\subseteq J$ by the definitions chosen at the beginning; hence, we can drop the randomness in $P_JX$ of all coordinates except for those in the spread set and apply the result. We again technically need to check that $\on{Spread}_\lambda(x_0)$ satisfies the conditions on \cref{prop:anticoncentration}, which we have already implicitly verified before. Overall, we deduce
\[\mb{P}_{X,X'}[|\sang{x_0,P_JX}-w|\le\epsilon_0^{-3/2}\epsilon\wedge\mc{E}_U|A]\le C\epsilon_0^{-3/2}\epsilon + 2^{-\lambda n/C}.\]

Finally, tracing it all back, we have
\[\mb{P}[\mc{E}]\le\mb{P}[\mc{E}_K^c] + 2p_0^{1/2} + p_1\le 2e^{-cn}+C\epsilon_0^{1/2} + C\epsilon_0^{1/4} + C\epsilon_0^{-3/4}\epsilon^{1/2} + 2^{-\lambda n/C}\]
by sub-Gaussian concentration of the operator norm of $A$ (\cite[Lemma~2.3]{Ver14}). Now choosing $\epsilon_0 = \epsilon^{1/2}$ and $\lambda = 1/\sqrt{n}$, which is the biggest permitted by \cref{thm:unstructured-kernel}, we are done.
\end{proof}

\end{document}